\documentclass{amsart}
\usepackage{amssymb,amsmath,amsthm,amscd}
\newtheorem{theorem}{Theorem}

\newtheorem{proposition}{Proposition}
\newtheorem{definition}{Definition}
\newtheorem{corollary}{Corollary}
\newtheorem{notation}{Notation}

\newtheorem{example}{Example}

\begin{document}

\title{Differential Brauer Monoids}
\author{Andy~R. Magid}
\address{Department of Mathematics\\
        University of Oklahoma\\
        Norman OK 73019\\
}
\email{amagid@ou.edu}
\dedicatory{Dedicated to the memory of Professor Raymond T. Hoobler}
\maketitle

\begin{abstract}
The differential Brauer monoid of a differential commutative ring $R$ is defined. Its elements are the isomorphism classes of differential Azumaya $R$ algebras with operation from tensor product subject to the relation that two such algebras are equivalent if matrix algebras over them, with entry-wise differentiation, are differentially isomorphic. The fine Bauer monoid, which is a group, is the same thing without the differential requirement. The differential Brauer monoid is then determined from the fine Brauer monoids of $R$ and $R^D$ and the submonoid of the Brauer monoid whose underlying Azumaya algebras are matrix rings.
\end{abstract}

\section*{Introduction} Let $R$ be a commutative ring with a derivation $D_R$, or just $D$ if it is clear from context. A \emph{differential  $R$ algebra} is an $R$ algebra $A$ with a derivation $D_A$ which extends the derivation $D_R$ on $R$. If $A$ is a differential $R$ algebra, an $A$ module $M$ is a \emph{differential $A$ module} if there is an additive endomorphism $D_M$ of $M$ satisfying $D_M(am)=D_A(a)m+aD_M(m)$ for all $a \in A$ and $m \in M$. If  $A$ and $C$ are differential $R$ algebras, an $R$ algebra homomorphism $f :A \to C$ is \emph{differential} provided $D_C f= fD_A$. Thus we may speak of isomorphism classes of differential algebras. If $A$ and $B$ are differential algebras, so is $A \otimes_R B$, where the derivation of the tensor product  is $D_A \otimes 1 + 1 \otimes D_B$. The set of isomorphism classes is closed under the operation induced from tensor product, forming a commutative monoid  with identity the class of $R$.

This paper is concerned with differential $R$ algebras which are Azumaya as $R$ algebras, or \emph{differential Azumaya algebras}. The monoid of isomorphism classes of differential Azumaya algebras under the operation induced from tensor product is denoted $MA^\text{diff}(R)$. Inspired by definitions of the Brauer group, we call two differential Azumaya algebras $A$ and $B$ \emph{equivalent} provided there are integers $m$ and $n$ such that the algebras $M_m(A)$ and $M_n(B)$, with entrywise derivations, are isomorphic differential Azumaya $R$ algebras. This relation is an equivalence relation compatible with the operation in the monoid $MA^\text{diff}(R)$; we denote the resulting monoid of equivalence classes $BM^\text{diff}(R)$ and call it the \emph{differential Brauer monoid of $R$}.

We recall that the Brauer group $Br(S)$ of a commutative ring $S$ (no derivation) is made up of equivalence classes of Azumaya algebras over the ring, with the operation induced from the tensor product of algebras, where two Azumaya $S$ algebras $A$ and $B$ are \emph{Brauer equivalent} if there are finitely generated projective $S$ modules $P$ and $Q$ such that there is an $S$ algebra isomorphism $A \otimes_S \text{End}_S(P) \cong B \otimes_s \text{End}_S(Q)$.  This definition of Brauer equivalence was introduced by Auslander and Goldman in 1960 \cite{ag} (and anticipated for local rings by Azumaya nine years earlier) and generalized the construction of the Brauer group of a field.  One could also consider the finer equivalence relation where $P$ and $Q$ are free: that is, $A$ and $B$ are \emph{fine Brauer equivalent} if for integers $m$ and $n$ there is an $S$ algebra isomorphism $A \otimes_S M_m(S) \cong B \otimes_S M_n(S)$. (This isomorphism could be rewritten as $M_m(A) \cong M_n(B)$; we will call this \emph{matrix Brauer equivalence}.) The set of equivalence classes under the finer relation, which we denote $Br^*(S)$, is also a group; see Proposition \ref{Brauermonoid}.

These three equivalence relations have differential analogues for the differential ring $R$.  Differential Azumaya algebras $A$ and $B$ are said to be equivalent if there are differential Azumaya $R$ algebras $C$ and $D$ and a differential isomorphism  $A \otimes_R C \cong B \otimes_R D$ where: 1) $C$ and $D$ as $R$ algebras are endomorphism rings of projective modules (\emph{differential Brauer equivalence}); or 2) $C$ and $D$ as $R$ algebras are endomorphism rings of free modules (\emph{differential fine Brauer equivalence}); or 3) $C$ and $D$ are matrix algebras with entry-wise derivation (\emph{differential matrix Brauer equivalence}).  The set of equivalence classes in each of the three cases has an operation induced from tensor product: the resulting monoids are  quotient monoids of $MA^\text{diff}(R)$. As we will see, in cases 1) and 2) the equivalence relation collapses differential distinctions and just recovers $Br(R)$ (case 1)) and $B^*(R)$ (case 2)), while case 3), which produces $BM^\text{diff}(R)$,  preserves differential information.  Note: the reason the differential distinctions collapse in cases 1) and 2) is that the equivalence relation allows all possible differential structures on the endomorphism rings $C$ and $D$, while in case 3) there is a designated differential structure (entrywise derivation) on each matrix ring.

However, the equivalence classes in case 3) need not be a group. 

In 1967 Bass \cite{b} showed how the Brauer group of a commutative ring could be defined using the K-theory of the category of Azumaya algebras over the ring. Building on this latter idea, Juan and the author defined a Brauer group for differential Azumaya algebras over a differential field $F$ \cite{jm}: we took $K_0$ of the category of differential central simple algebras over $F$ and passed to the quotient by the subgroup generated by the matrix algebras over $F$ with entry-wise derivation, terming this quotient the differential Brauer group. The same definition applied to a differential ring $R$ defines what could be called the differential Brauer group of $R$. Although this differential Brauer group maps (surjectively) to the usual Brauer group by forgetting the derivations, its origin in K-theory presents difficulties. Unlike in the Brauer group of a commutative ring, not every element of this differential Brauer group is the image of an isomorphism class of a differential Azumaya algebra; this is due to the fact that the inverse of the class of a differential Azumaya algebra need not be the class of a differential Azumaya algebra. By staying with monoids, instead of passing to $K$ groups, this difficulty is avoided. The idea of using monoids of isomorphism classes instead of groups is found, in a different context, in \cite[Section 3 p. 297]{i}

By this approach, we seek to isolate in algebraic structure the additional contribution of ``differential" to the classification of differential Azumaya algebras, by separating out the contributions from the classification of the underlying Azumaya algebras captured in $Br^*(R)$ as well as the contribution from the ring of constants captured in $Br^*(R^D)$.

This process begins with the homomorphism  $BM^\text{diff}(R) \to Br^*(R)$. By Knus's Theorem \cite[Th\'eor\`eme 3, p. 639]{k} if $A$ is a differential Azumaya algebra which is a subalgebra of an Azumaya algebra $B$, then there is a derivation of $B$ which extends that of $A$. In particular, 
every Azumaya algebra $A$ over the differential ring $R$ has a derivation which extends that of $R$ so the above homomorphism is surjective.

Although $BM^\text{diff}(R)$ is not a group in general, we can look at its subgroup of invertible elements. As we will see, this subgroup is an isomorphic image of $Br^*(R^D)$. 

The foundational results on differential Azumaya algebras on which these facts depend are presented in Section \ref{basics} below, while the results on the monoids are covered in Section \ref{monoids}. Section \ref{fields} looks at the case when $R$ is a differential field and at other examples.

We retain throughout the terminology and notations of this introduction. In particular, $R$ is always a commutative differential ring.

This work is a collaborative project of the author and the late Raymond T. Hoobler, undertaken with the idea of publishing it on the 50th anniversary of our first collaboration \cite{hm}. Except for the fact that Ray is no longer here to correct the proofs, in both senses, he is in every sense a coauthor.

\section{Differential Azumaya Algebras} \label{basics}

It is a basic property of an Azumaya algebra $A$  that derivations which are trivial on the center are inner.  For $a \in A$, we let $I_a$ denote inner derivation by $a$. This means that if we have one derivation of an Azumaya algebra we in effect have them all:

\begin{proposition} \label{AllDerivations} Let $A$ be a differential Azumaya $R$ algebra with derivation $D_A$. If $D_1$ is a derivation of $A$ which agrees with $D_R$ on $R$ then $D_1=D_A + I_a$ for some $a \in A$.
\end{proposition} 

\begin{proof} Since  $D_1$ and $D_A$ on $R$ are $D_R$, $D_1-D_A$ is a derivation of $A$ which is trivial on $R$ so $D_1-D_A=I_a$  for some $a \in A$. Thus $D_1=D_A + I_a$.
\end{proof}

Two important special cases of Proposition \ref{AllDerivations} occur when $A$ is the endomorphism ring of a projective module and when $A$ is a matrix ring.

We recall that a differential projective $R$ module is a differential $R$ module which is finitely generated and projective as an  $R$ module

\cite[Definition 1, p. 4338]{jm}.
If $P$ is such a differential projective module then $\text{End}_R(P)$ is a differential Azumaya algebra with induced derivation $D(S)=D_PS - SD_P$.

\begin{corollary} \label{EndoOfProjective} Let $P$ be a finitely generated projective $R$ module and let $D$ be a derivation of $\text{End}_R(P)$. Then there is a differential structure on $P$ such that the derivation induced from $D_P$ is $D$.
\end{corollary}

\begin{proof} By \cite[Theorem 2, p. 4341]{jm}, there is a differential structure $D_0$ on $P$. Let $D_1$ be the induced derivation on $\text{End}_R(P)$. By Proposition \ref{AllDerivations} $D=D_1+I_T$ for some $T \in \text{End}_R(P)$. Consider the additive endomorphism $D_0+T$. It is easy to check, or we can cite \cite[p. 4338]{jm}, that $D_0+T$ is also a differential structure on $P$, and the derivation it induces on $\text{End}_R(P)$  sends $S$ to $(D_0+T)S-S(D_0+T)=D_0S-SD_0+TS-ST=D_1(S)+I_T(S)=D(S)$.
\end{proof}

Let $P$ be a projective $R$ module. If an  $R$ algebra $A$ is a subalgebra of $\text{End}_R(P)$, we can view $P$ as a faithful $A$ module via the endomorphisms given by left multiplication, and conversely. Suppose $A$ is also Azumaya over $R$. Then we have the following consequence of Corollary \ref{EndoOfProjective}:

\begin{corollary} \label{projectivemodulesdifferential} Let $A$ be a differential Azumaya $R$ algebra, and let $P$ be an $A$ module which is finitely generated and projective as an $R$ module. Then there is a differential structure on $P$ making $P$ a differential $A$ module.
\end{corollary} 

\begin{proof} As noted, $P$ gives an algebra embedding $A \to \text{End}_R(P)$, which we regard as an inclusion. Both $\text{End}_R(P)$ and its subalgebra $A$ are Azumaya over $R$. 

By \cite[Th\'eor\`eme 3, p. 639]{k}, the derivation of $A$ extends to a derivation of $\text{End}_R(P)$, making the inclusion one of differential algebras. By Corollary \ref{EndoOfProjective}, there is a differential $R$ module structure on $P$  which induces the given derivation on $\text{End}_R(P)$. Then for  $p \in P$ and $a \in A$ (or indeed for any endomorphism of $P$) $D(ap)=D(a)p+aD(p)$.
\end{proof}

If $M_n(R)$ is a matrix ring over $R$, then applying $D_R$ entry-wise gives a derivation, which we denote $(\cdot )^\prime$. And once we have one derivation, by Proposition \ref{AllDerivations}, we have then all.

\begin{corollary} \label{matrixderivation} Let $D$ be a derivation of $M_n(R)$. Then there is a matrix $X \in M_n(R)$ such that $D(Y)=Y^\prime + XY-YX$ for $Y \in M_n(R)$. Conversely, any matrix $Z$ defines a derivation $\partial_Z$ of $M_n(R)$ by the formula $\partial_Z(Y)=Y^\prime +ZY-YZ$.
\end{corollary} 

\begin{proof} By Proposition \ref{AllDerivations}, $D=(\cdot )^\prime + I_X$ for some $X \in M_n(R)$, which is the first assertion. The second follows since the sum of the two derivations $(\cdot )^\prime$ and $I_Z$ is a derivation which restricts to $D_R$ on $R$.
\end{proof}

\begin{notation} \label{derivationofmatrix} For $A \in M_n(R)$ denote $M_n(R)$ with the derivation $\partial_A$ by $(M_n(R), A).$
\end{notation}

An important special case is where the matrix $A$ in Corollary \ref{matrixderivation} is zero: the derivation $\partial_0$ is just another name for $(\cdot )^\prime$. Because $A=0$, we will refer to this as the \emph{trivial} derivation on $M_n(R)$.

The matrix algebra $M_n(R)$ is isomorphic to the endomorphism ring of the free module $R^n$ (column vectors). By Corollary \ref{EndoOfProjective}, the differential structure $(M_n(R), A)$ is induced by a differential structure on $R^n$. It is easy to see, using the proof of the Corollary as a guide, that in this structure the derivative of $x$ is $x^\prime + Ax$, where $x^\prime$ means coordinate wise derivation by $D_R$. We abbreviate this structure $(R^n,A)$. Note that this is the dual of the similar structure in \cite[p. 4338]{jm2}, so we can use the calculations from \cite{jm2} after changing rows to columns. In particular, if $Y \in M_n(R)$ is invertible and $Y^\prime=-AY$ then left multiplication by $Y$ defines a differential isomorphism $(R^n, 0) \to (R^n, A)$. Taking endomorphism rings, this means that conjugation by $Y$ defines a differential isomorphism $(M_n(R), 0) \to (M_n(R), A)$.

Conversely, suppose we have a differential isomorphism $(M_n(R), 0) \to (M_n(R), Z)$ and suppose the isomorphism is given by conjugation by the invertible matrix $T$. Then $T$ satisfies a differential equation of the form $T^\prime = -(Z +cI)T$ where $c \in R$. (In \cite[Proposition 2, p.1923]{jm} we show how to make $T$, there called $H$, determinate, using the fact that the base ring $R$ is a field of characteristic zero. In the general case we need to keep the possibility of any $c$.)

Every Azumaya algebra appears as a tensor factor in a matrix ring: this follows from Bass's Theorem, \cite[Proposition 4.6, p. 476]{b2}, that if $P$ is a faithful finitely generated projective $R$ module then there is a faithful finitely generated projective $R$ module $Q$ such that $P \otimes Q$ is free. As $A$ is a subalgebra and tensor factor of $\text{End}_R(A)$ by left multiplications and since $A$ is faithful and finitely generated projective as an $R$ module there is a $Q$ as above, so 
$\text{End}_R(A) \otimes \text{End}_R(Q)$ is a matrix algebra. However for a differential Azumaya algebra to appear as a tensor factor of a matrix ring with trivial derivation is a very restrictive condition, as we see in Theorem \ref{factorinduced}.

The proof of this Theorem will use the result, \cite[Theorem 1, p. 4340]{jm2}, that a differential projective module $P$ is a differential direct summand of $R^n$, where the latter has the component-wise differential structure, if and only if $P$ is of the form $R \otimes_{R^D} P_0$ where $P_0$ is a finitely generated projective $R^D$ module. The differential structure on $R \otimes_{R^D}P_0$ is given by $D(r \otimes p)=D(r) \otimes p$. It is further shown in \cite[ Lemma 1, p. 4340]{jm2}     that   if $P = R \otimes_{R^D} P_0$ for finitely generated projective $P_0$ then $P_0=P^D$.

\begin{theorem} \label{factorinduced} Let $A$ be a differential Azumaya algebra and suppose there is a differential Azumaya algebra $B$ such that $A \otimes_R B$ is isomorphic, as a differential $R$ algebra,  to $M_n(R)$ with the trivial derivation. Then $A^D$ is an Azumaya algebra  over $R^D$ (in fact $A^D \otimes B^D$ is isomorphic to $M_n(R^D)$) and $R \otimes_{R^D}A^D \to A$ by $r \otimes a \mapsto ra$ is an isomorphism. Conversely, if $A_0$ is an Azumaya algebra over $R^D$ then the Azumaya $R$ algebra $R \otimes_{R^D} A_0$ is a differential tensor factor of an $R$ matrix algebra with trivial derivation.
\end{theorem}

\begin{proof} We assume both $A$ and $B$ are subalgebras of $M_n(R)$. Consider the map $p_{11}: M_n(R) \to R$ given by projection on the entry of the first row and column of a matrix. Because the derivation on  $M_n(R)$ is the trivial one, $p_{11}$ is a homomorphism of differential modules which sends the identity to one. Let $p$ be the restriction of $p_{11}$ to $B$. The kernel $B_0$ of $p$ is then a differential module, and since $p$ is onto we have $B = R \oplus B_0$ as differential modules. Thus as differential modules $M_n(R)$ is $(A \otimes_R R) \oplus (A \otimes_R B_0)$. In particular, $A$ is a differential direct summand of $M_n(R)$, which as a differential module is $R$ free with coordinate-wise derivation. By \cite[Theorem 1, p. 4340]{jm2}, this means that $A^D$ is a finitely generated projective $R^D$ module and $R \otimes_{R^D}A^D \to A$ is an isomorphism. The same applies to $B^D$. Note that both $A^D$ and $B^D$ are subalgebras of $M_n(R)^D=M_n(R^D)$. To see that $A^D$ is Azumaya, we consider the multiplication map $m_0: A^D \otimes B^D \to M_n(R^D)$. 
Then $1 \otimes m_0$, under the multiplication isomorphisms $R \otimes A^D \to A$ and $R \otimes B^D \to B$, becomes the multiplication map $m: A \otimes B \to M_n(R)$, which is a differential algebra isomorphism. Since $1 \otimes m_0: R \otimes_{R^D} (A^D \otimes B^D) \to M_n(R)$ is a differential algebra isomorphism, it induces an isomorphism of the constants of its domain, namely $A^D \otimes B^D$, and the constants of its range, $M_n(R)^D=M_n(R^D)$.
Thus $A^D$ is an Azumaya $R^D$ algebra. The final statement is a consequence of the fact, applied to $R^D$,  that every Azumaya algebra is a tensor factor of a matrix algebra, and the fact that $R \otimes M_n(R^D)$ is isomorphic to $M_n(R)$ with the trivial derivation as a differential $R$ algebra.
\end{proof}

If $A$ is a differential Azumaya $R$ algebra, and $P$ is an $A$ module which is finitely generated and free of rank $n$ as an $R$ module, by Corollary \ref{projectivemodulesdifferential} $P$ has a differential module structure as an $A$ module. If we regard $R^n$ (column $n$ tuples) as a left $M_n(R)$ module via matrix multiplication, then the differential structure of $P$ induces a derivation of $M_n(R)$, and the differential inclusion $A \to M_n(R)$ implies that $A$ is a differential tensor factor of $M_n(R)$, the other factor being the commutator of $A$, which is also a differential subalgebra. By Theorem \ref{factorinduced}, if the  derivation of $M_n(R)$ is trivial, $A$ is induced up from $R^D$. Hence:

\begin{corollary} \label{constantbasis} Let $A$ be a differential Azumaya $R$ algebra. Then $A$ has a faithful differential module $P$, finitely generated and projective over $R$, with an $R$ basis of constants (and hence  free), if and only if $A$ is isomorphic, as a differential algebra,  to $R \otimes_{R^D} A_0$ for some Azumaya $R^D$ algebra $A_0$.
\end{corollary}

\begin{proof}  If  $A$ has a faithful differential module $P$, finitely generated and projective over $R$, with an $R$ basis of constants. then $A$ is a differential subalgebra of $M_n(R)$ with the trivial derivation, and hence of the desired form. Convesely, if $A$ is isomorphic to $R \otimes_{R^D} A_0$ for some Azumaya $R^D$ algebra $A_0$ and if $P_0$ is a faithful $A_0$ module, finitely generated and free as an $R^D$ module, then $P=R \otimes P_0$ is a faithful $A$ module, finitely generated and projective over $R$, with an $R$ basis of constants.
\end{proof}

\section{Brauer Monoids}  \label {monoids}

Throughout this section we will be concerned with monoids of isomorphism classes and various quotients thereof. To fix terminology, by a monoid $M$ we mean a commutative semigroup with identity $1$, with the operation written multiplicatively. If $N \subseteq M$ is a submonoid, by $M/N$ we mean the set of equivalence classes on $M$ under the equivalence relation $m_1 \sim m_2$ if there are $n_1, n_2 \in N$ such that $m_1n_1=m_2n_2$. Products of equivalent elements are equivalent, which defines a commutative operation on $M/N$ by multiplying representatives of equivalence classes, and the equivalence class of the identity is an identity. Thus $M/N$ is a monoid.  Note that we can have $M/N=1$ but $N \neq M$, for example $\mathbb Q^*/\mathbb Z^\times$. If $P \subset Q$ are submonoids of $M$ then $Q/P$ is a submonoid of $M/P$ and $(M/P)/(Q/P)$ is isomorphic to $M/Q$ (``third isomorphism theorem") for monoids.

The set of invertible elements $U(M)$ of $M$ is its maximal subgroup. The set elements $a \in M$ whose equivalence classes in $M/N$ are invertible is $\{ a \in M \vert \exists b \in M, n \in N \text{ such that } abn \in N\}$. 

We will be considering the following monoids:

\begin{definition} \label{themonoids} 
$MA(R)$ denotes the set of isomorphism classes of Azumaya $R$ algebras with the operation induced by tensor product and the identity the isomorphism class of $R$.

$MA^\text{diff}(R)$  denotes the set of isomorphism classes of differential Azumaya $R$ algebras with the operation induced by tensor product and the identity the isomorphism class of $R$.

$ME(R)$ denotes the submonoid of $MA(R)$ whose elements are isomorphism classes of endomorphism rings of faithful finitely generated projective $R$ modules.

$ME^\text{diff}(R)$ denotes the submonoid of $MA^\text{diff}(R)$ whose elements are isomorphism classes of differential Azumaya algebras which as $R$ algebras are endomorphism rings of faithful finitely generated projective modules.

$MM(R)$ denotes the submonoid of $MA(R)$ whose elements are isomorphism classes of matrix algebras.

$MM^\text{diff}(R)$ denotes the submonoid of $MA^\text{diff}(R)$ whose elements are isomorphism classes of Azumaya algebras of the form $(M_n(R), Z)$.

$MM_0^\text{diff}(R)$ denotes the submonoid of $MA^\text{diff}$ whose elements are isomorphism classes of Azumaya algebras of the form $(M_n(R), 0)$.

$BM(R)$ and $Br^*(R)$ both denote $MA(R)/MM(R)$

$BM^\text{diff}(R)$ denotes $MA^\text{diff}(R)/MM_0^\text{diff}(R)$

$EBM^\text{diff}(R)$ denotes  $ME^\text{diff}(R)/MM_0^\text{diff}(R)$

$MBM^\text{diff}(R)$ denotes $MM^\text{diff}(R)/MM_0^\text{diff}(R)$

$BM(R)$ is called the \emph{Brauer monoid} of $R$.

$BM^\text{diff}(R)$ is called the \emph{differential Brauer monoid} of $R$.

$EBM^\text{diff}(R)$ is called the \emph{endomorphism differential Brauer monoid} of $R$.

$MBM^\text{diff}(R)$ is called the \emph{matrix diffferential Brauer monoid} of $R$.
\end{definition}

All of the monoids in Definition \ref{themonoids} are functors.

All of the monoids of equivalence classes discussed in the introduction can be viewed as quotient monoids: the set of equivalence classes of Azumaya algebras under the relation of Brauer equivalence (the Brauer group $Br(R)$) is the quotient $MA(R)/ME(R)$; and the set of equivalence classes under fine Brauer equivalence, which we also termed matrix Brauer equivalence, (the fine Brauer group $Br^*(R)$) is the quotient $BM(R)=MA(R)/MM(R)$. 

The set of equivalence classes of differential Azumaya algebras under differential Brauer equivalence is $MA^\text{diff}(R)/ME^\text{diff}(R)$; the set of equivalence classes under
differential fine Brauer equivalence is $MA^\text{diff}(R)/MM^\text{diff}(R)$; and the set of equivalence classes of differential Azumaya algebras under the relation of differential matrix Brauer equivalence (the differential Brauer monoid $BM^\text{diff}(R)$) is $MA^\text{diff}(R)/MM_0^\text{diff}(R)$.

Since matrix rings are a special case of endomorphism rings of projective modules, fine Brauer equivalence implies Brauer equivalence, so we have a surjection $Br^*(R) \to Br(R)$; this is a group homomorphism with an identifiable kernel.

\begin{proposition} \label{Brauermonoid} The Brauer monoid $BM(R)=Br^*(R)$ is a group. The group homomorphism $Br^*(R) \to Br(R)$ is surjective with kernel $ME(R)/MM(R)$. In particular, the latter is a group.
\end{proposition}

The proof of the proposition will use the fact, already mentioned, that an Azumaya $R$ algebra $A$ is a tensor factor of a matrix algebra. 

\begin{proof}  As remarked above, the invertible elements of $Br^*(R)$ are the classes of the Azumaya algebras $A$ for which there exists an Azumaya algebra $B$ and matrix algebras $M_1$ and $M_2$ such that $A \otimes B \otimes M_1$ is isomorphic to $M_2$. Since any Azumaya algebra is a tensor factor of a matrix algebra, this condition holds. Thus $Br^*(R)$ is a group. If the Azumaya algebra $A$ is the endomorphism ring of a projective module $P$, by the theorem of Bass previously quoted there is a projective module $Q$ such that, if $B$ denotes the endomorphism ring of $Q$, $A \otimes_R B$ is a matrix ring. So the class of $B$ is the inverse of the class of $A$ in $Br^*(R)$. This shows that $ME(R)/MM(R)$ is a subgroup of $MA(R)/MM(R)=Br^*(R)$. Since matrix algebras are trivial in the Brauer group, the monoid surjection $MA(R) \to Br(R)$ factors through $Br^*(R)$. Suppose the Azumaya algebra $A$'s isomorphism class in $MA(R)$ maps to a class in $Br^*(R)$ which goes to the identity in $Br(R)$. That is, suppose the class of $A$ is trivial in the Brauer group of $R$. Then $A$ is the endomorphism ring of a projective, and so its isomorphism class in $MA(R)$ lies in $ME(R)$. It follows that  $ME(R)/MM(R)$ is the kernel.
\end{proof}

Proposition \ref{Brauermonoid} shows that the Brauer monoid of a commutative ring differs from the Brauer group of the ring by an identifiable subgroup. If this subgroup is trivial, then $Br^*(R)$ and $Br(R)$ coincide. To describe this situation, we introduce some terminology:

\begin{definition} \label{stable} An $R$ algebra  $A$ is said to be a \emph{stably matrix}  algebra if there are $n$ and $m$ such that  $A \otimes M_n(R)$ is isomorphic to $M_m(R)$. A differential $R$ algebra  $A$ is said to be a \emph{differentially stably matrix} algebra if there are $n$ and $m$ such that  $A \otimes (M_n(R),0) $ is differentially isomorphic to $(M_m(R), 0)$. A faithful projective $R$ module $P$ is said to be \emph{stably tensor free} is there are $n$ and $m$ such that $P \otimes R^n$ is isomorphic to $R^m$.
\end{definition} 

A differential Azumaya algebra $A$ represents the identity element of $BM^\text{diff}(R)$ if and only if $A$ is a differentially stably matrix algebra. If so, and if $A \otimes (M_n(R),0) $ is differentially isomorphic to $(M_m(R), 0)$, then by Theorem \ref{factorinduced} $A$ is differentially isomorphic to $R \otimes A^D$. This implies that $A^D$ is a projective $R^D$ module of the same rank as that of $A$ over $R$. If the ranks of $A$ and $A^D$ differ, then the class of $A$ in $BM^\text{diff}(R)$ is not the identity element.

\begin{corollary} \label{BrauermonidequalsBrauergroup} The map $Br^*(R) \to Br(R)$ is an isomorphism if and only if endomorphism rings of faithful projective $R$ modules are stably matrix algebras. If all faithful projective $R$ modules are free, or even stably tensor free, then all their endomorphism rings are stably matrix algebras.
\end{corollary} 

It is not in general the case that the differential Brauer monoid is a group. 

Our next result can be understood as saying that the differential Brauer monoid differs from the Brauer monoid by the matrix differential Brauer monoid, and it differs from the Brauer group by the endomorphism differential Brauer monoid. 

\begin{theorem} \label{monoidkernel} The monoid surjection $BM^\text{diff}(R) \to BM(R)$ induces isomorphisms $BM^\text{diff}(R)/MBM^\text{diff}(R) \to BM(R)=Br^*(R)$ and $BM^\text{diff}(R)/EMB^\text{diff}(R)\to Br(R)$.
Thus $MA^\text{diff}(R)/MM^\text{diff}(R) \cong Br^*(R)$ and $MA^\text{diff}(R)/EMB^\text{diff}(R) \cong Br(R)$.
\end{theorem}

\begin{proof}  For the purposes of the proof, we will write $\cong$ for $R$ algebra isomorphism and $\cong_\text{diff}$ for differential $R$ algebra isomorphism. Suppose $A$ and $B$ are differential Azumaya $R$ algebras whose images are equal in $BM(R)$. Let $C$ be an Azumaya $R$ algebra whose image in $BM(R)$ is the inverse of the class of $B$. Then there are  matrix $R$ algebras $M_c$ and $M_d$ such that $(1)$ $B \otimes C \otimes M_c \cong M_d$ (since $C$ is the inverse of $B$) and, since $A$ and $B$ are equal in the Brauer monoid, there are also matrix algebras $M_a$ and $M_b$ such that $(2)$ $A \otimes C \otimes M_a \cong M_b$.  Now $A$ and $B$ already have differential structures. Put a differential structure on $C$ and $M_c$, and use the isomorphism $(1)$ to give $M_d$ a differential structure so that $(1)$ is a differential isomorphism. Put a differential structure on $M_a$, and use the differential structure already placed on $C$ and the isomorphism $(2)$ to give $M_b$ a differential structure so that $(2)$ is a differential isomorphism. We tensor $(2)$ with $B \otimes M_c$ to obtain 
$$A \otimes C \otimes M_a \otimes B \otimes M_c \cong_\text{diff} M_b \otimes B \otimes M_c.$$
Then we use $(1)$ to replace $B \otimes C \otimes M_c$ on the left with $M_d$ to obtain
$$ A \otimes M_d \otimes M_a \cong_\text{diff} B \otimes M_a \otimes  M_c$$
which shows that $A$ and $B$ are fine differential Brauer equivalent. Thus the surjection $$BM^\text{diff}(R) \to BM(R)$$ induces an injection, hence an isomorphism, $$BM^\text{diff}(R)/MBM^\text{diff}(R)\to BM(R).$$

Now suppose that $A$ and $B$ are differential Azumaya $R$ algebras whose images are equal in $Br(R)$. Let $C$ be an Azumaya $R$ algebra representing the inverse of $B$ in $Br(R)$. There are endomorphism rings $E_1$ and $E_2$ of projective $R$ modules such that $(3)$ $A \otimes C \cong E_1$ and $(4)$ $B \otimes C\cong E_2$. Put a differential structure on $C$, and use $(3)$ and $(4)$ to put differential structures on $E_1$ and $E_2$ such that $(3)$ and $(4)$ become differential isomorphisms. Tensor $(3)$ with $B$: this gives a differential isomorphism $A \otimes C \otimes B \cong_\text{diff} E_1 \otimes B$, then use $(4)$ to replace $C \otimes B$ with $E_2$ to obtain a differential isomorphism
$$ A \otimes E_2 \cong_\text{diff} B \otimes E_1$$ which shows $A$ and $B$ are differential Brauer equivalent. Thus the surjection $$BM^\text{diff}(R) \to Br(R)$$ induces an injection, hence an isomorphism, 
$$BM^\text{diff}(R)/ME^\text{diff}(R)\to Br(R).$$ The final assertions are restatements of the first two using the third isomorphism theorem for monoids.

\end{proof}

The final isomorphisms  of Theorem \ref{monoidkernel} say that differential Azumaya algebras are differential fine Brauer equivalent if their underlying algebras are fine Brauer equivalent, and they are differential Brauer equivalent if their underlying algebras are Brauer equivalent. That is, neither of these differential equivalence relations make any distinctions based on differential structure. Differential matrix Brauer equivalence, on the other hand, does make distinctions based on differential structure, as is shown below.  Since matrix Brauer equivalence impacts differential matrix Brauer equivalence, we keep track of their connection.

As a consequence of Theorem \ref{monoidkernel}, we know that, modulo the matrix differential Brauer monoid, the differential Brauer monoid and the Brauer monoid coincide, so that in particular (as noted in the following corollary) if the matrix differential Brauer monoid is trivial the differential Brauer monoid and the Brauer monoid coincide. In this sense, the matrix differential Brauer monoid captures the additional structure of the differential case.

\begin{corollary} \label{MBMtrivial} If $MBM^\text{diff}(R)=1$ then $BM^\text{diff}(R) =BM(R)$.
\end{corollary}

It is also true that if the Brauer monoid of $R$ is trivial then the differential Brauer monoid of $R$ coincides with the matrix differential Brauer monoid of $R$. Since as noted above, the triviality of a quotient of a monoid by a submonoid does not imply that the monoid and submonoid are equal, this is not quite a corollary of the statement of  Theorem \ref{monoidkernel}, but it follows from the same techniques used in its proof:

\begin{corollary} \label{trivialBM} If $BM(R)=1$ then $MBM^\text{diff}(R)=BM^\text{diff}(R)$
\end{corollary}

\begin{proof} If $BM(R)=1$ then for any Azumaya $R$ algebra $A$ there are $p$ and $q$ such that $A \otimes M_p(R) \cong M_q(R)$. If $A$ is a differential Azumaya algebra, so is $A \otimes (M_p(R), 0)$ and hence for suitable $Z$ $A \otimes (M_p(R),0) \cong_\text{diff} (M_q(R), Z)$. Since the class of $(M_q(R), Z)$  is in $MBM^\text{diff}(R)$, so is the class of $A$.
\end{proof}

The hypothesis of Corollary \ref{trivialBM} is satisfied, for example, using Proposition \ref{Brauermonoid} and Corollary \ref{BrauermonidequalsBrauergroup}, when $Br(R)=1$ and all faithful projective $R$ modules are free. 

Theorem \ref{monoidkernel} and Corollary \ref{trivialBM} are about the contribution of $BM(R)$ to $BM^\text{diff}(R)$. Our next result is about the contribution of $BM(R^D)$.

\begin{theorem} \label{unitsinduced} The map $BM(R^D) \to BM^\text{diff}(R)$ by $A_0 \mapsto R \otimes A_0$ is an isomorphism from $BM(R^D)$ to the group of units of $BM^\text{diff}(R)$.
\end{theorem}

\begin{proof} By Proposition \ref{Brauermonoid} $BM(R^D)$ is a group, and so the monoid homomorphism carries it to the group of units $U$ of $BM^\text{diff}(R)$. Suppose $A$ is a differential Azumaya $R$ algebra whose class in $BM^\text{diff}(R)$ is a unit. This means there is a differential Azumaya algebra $B$ and $p$, $q$ such that $A \otimes B \otimes (M_p(R), 0)$ is differentially isomorphic to $(M_q(R), 0)$. Thus $A$ ia a tensor factor of a matrix algebra with trivial derivation. By Theorem \ref{factorinduced} this means that $A^D$ is an Azumaya $R^D$ algebra and that $A$ is differentially isomorphic to $R \otimes A^D$. Thus $A$ lies in the image of $BM(R^D)$. To see that the onto group homomorphism $BM(R^D) \to U$ is an injection, let $A_0$ be an Azumaya algebra representing an element of the  kernel. Then there are $p$ and $q$ such that $(R \otimes A_0) \otimes (M_p(R), 0)$ is differentially isomorphic to $(M_q(R), 0)$. By Theorem \ref{factorinduced}, taking constants this differential isomorphism gives an $R^D$ isomorphism of $A_0 \otimes M_p(R^D)$ to $M_q(R^D)$ which means the class of $A_0$  is the identity in $BM(R^D)$.
\end{proof}

The kernel of the group homomorphism $\Phi:BM(R^D) \to BM(R)$ induced from the inclusion consists of classes represented by Azumaya $R^D$ algebras $A_0$ such that $R \otimes A_0$ is a stably matrix algebra, say $(R \otimes A_0) \otimes M_n(R) \cong M_m(R)$. Since $A_0 \otimes M_n(R^D)$ represents the same class as $A_0$, and $R \otimes (A_0 \otimes M_n(R^D))=(R \otimes A_0) \otimes M_n(R)$ is isomorphic to $M_m(R)$, we can replace $A_0$ with $A_0 \otimes M_n(R^D)$ and assume that $\Phi(A_0)$ is a matrix algebra. Now $\Phi$ factors as $BM(R^D) \to BM^\text{diff}(R) \to BM(R)$. The first map is injective by Theorem \ref{unitsinduced}, and by the remarks just noted the kernel of $\Phi$  is $BM(R^D) \cap MBM^\text{diff}(R)$, which, by Theorem \ref{unitsinduced} is the group of units of $MBM^\text{diff}(R)$.

An interesting special case occurs when $BM(R^D)=1$. Then Theorem \ref{unitsinduced} implies that $BM^\text{diff}(R)$ has no units except $1$, and in particular no torsion elements, in contrast with the situation of the Brauer group of $R$, which is torsion \cite[Corollary 11.2.5, p. 417]{f}.
When $BM^\text{diff}(R)$ has no units except $1$, every non-trivial element of $BM^\text{diff}(R)$ generates a cyclic submonoid of infinite order. In particular, if $BM^\text{diff}(R)$ is not trivial, it is infinite. 

When $BM^\text{diff}(R)$ has no units except 1, and $BM(R)$ is non--trivial, the monoid surjection $BM^\text{diff}(R) \to BM(R)$ can't split since by Proposition \ref{Brauermonoid} $BM(R)$ is a group. This implies that the existence of differential structures on Azumaya algebras can't be made canonical.

We also have the following consequences of Theorem \ref{monoidkernel}, Theorem \ref{unitsinduced}, Corollary \ref{MBMtrivial}, and Corollary \ref{trivialBM}:

\begin{corollary} \label{summary} If $BM(R^D)=BM(R)=1$ then $BM^\text{diff}(R)=MBM^\text{diff}(R)$ and has no units. If $MBM^\text{diff}(R)=1$ then $BM(R^D) \to BM(R)$ is an isomorphism.
\end{corollary}

\section{Fields and Examples} \label{fields} 

We consider the case that $R^D=C$ is an algebraically closed field of characteristic zero. Since $C$ is a field, $Br^*(C)=Br(C)$ and since $C$ is algebraically closed $Br(C)=1$. 

\begin{example} \label{C} We consider the case $R=C$ equipped with the zero derivation. \end{example}Then by Corollary \ref{summary} $BM^\text{diff}(C)=MBM^\text{diff}(C)$. Consider the element represented by $(M_n(C), Z)$. As a linear transformation on $M_n(C)$ we have $I_Z =L_Z-R_Z$ where $L_Z$ (respectively $R_Z$) means left (respectively right) multiplication by $Z$. Any $C$ polynomial satisfied by $Z$ is satisfied by both $L_Z$ and $R_Z$, and since $L_Z$ and $R_Z$ commute we conclude that every eigenvalue of $I_Z$ is a difference of eigenvalues of $Z$. For any $p$, the linear transformation $I_Z \otimes 1 + 1 \otimes I_0=I_Z \otimes 1$ on $M_n(C) \otimes M_p(C)$ will have the same eigenvalues as $I_Z$. It follows that the set of eigenvalues of $I_Z$ is an invariant of the class of $(M_n(C), Z)$ in $BM^\text{diff}(C)$, and these eigenvalues are among the set of differences of eigenvalues of $Z$. In particular, if $Z \in M_n(C)$ and $W \in M_m(C)$ are such that the  their eigenvalues differences are disjoint, then $(M_n(C), Z)$ and $M_m(C),W)$ represent different elements of $BM^\text{diff}(C)$. We conclude that the cardinality of $BM^\text{diff}(C)$ is that of $C$.

\begin{example} \label{C(x)} We consider the case that $R=C(x)$ with derivation $d/dx$. \end{example} Then by \cite{ab}, $Br(C(x))=1$, and $C(x)$ being a field $Br^*(C(x))=Br(C(x))=1$ as well. By Corollary \ref{summary}, $BM^\text{diff}(C(x))=MBM^\text{diff}(C(x))$. Consider the map $BM^\text{diff}(C) \to BM^\text{diff}(C(x))$. Let $Z=[z_{ij}] \in M_2(C)$ be the  matrix with $z_{12}=1$ and the rest of its entries zero. Calculation shows that $I_Z^2 \neq 0$ although  $I_Z^3=0$, 

If there were $p,q$ such that $(M_2(C),Z) \otimes (M_p(C),0)$ were differentially isomorphic to $(M_q(C),0)$ then $I_Z \otimes 1$ would become the derivation $I_0$ on $M_q(C)$, but the  since the square of the former is not zero and the square of the latter is, this is impossible. Thus the class of $(M_2(C),Z)$ is not trivial in $BM^\text{diff}(C)$. However for the class of  $C(x) \otimes (M_2(C), Z)=(M_2(C(x)),Z)$ in $BM^\text{diff}(C(x))$, there is a $Y=[y_{ij}]\in M_2(C(x))$ such that $Y^\prime=-ZY$, namely where $y_{11}=x$, $y_{21}=-1$, $y_{12}=1$ and $y_{22}=0$. As we remarked above in the discussion after 
Notation \ref{derivationofmatrix}, 
conjugation by $Y$ then defines a differential isomorphism from $(M_2(C(x)), 0) \to (M_2(C(x)),Z)$. So in $BM^\text{diff}(C(x))$, the class of $C(x) \otimes (M_2(C), Z) $ is trivial. 

Now let $\lambda \in C$ and let $Z(\lambda) \in M_2(C)$ be the matrix with diagonal elements $1+ \lambda$ and $1$, with $\lambda \neq 0$. For $A=\begin{bmatrix} a& b\\ c&d \end{bmatrix}$, $[Z(\lambda),A]=\begin{bmatrix} 0 & \lambda b\\ -\lambda c & 0 \end{bmatrix}$. The standard unit matrices $e_{ij}^{(2)}$, $i=1,2$, $j=1,2$ are eigenvectors of $[Z(\lambda), \cdot ]$, with eigenvalues $0, \lambda, -\lambda$. Let $e_{ij}^{(p)}$, $1 \leq i,j \leq p$ be the standard unit matrices for $M_p(C)$. Then $e_{ij}^{(2)} \otimes e_{kl}^{(p)}$ are a $C$ basis of $M_2(C) \otimes M_p(C)$ of eigenvectors of $[Z(\lambda) \otimes 1, \cdot]$, the eigenvalues being $0, \lambda, -\lambda$. 

In the following, we are identifying $M_2(C(x)) \otimes M_p(C(x))$ with $M_{2p}(C(x))$ and regarding $Z(\lambda) \otimes 1$ as an element of this matrix ring in order to use Notation \ref{derivationofmatrix}.

Consider  $A=(M_2(C(x)) \otimes M_p(C(x)), Z(\lambda) \otimes 1)$; for convenience we write the derivation $D$. Then $D(e^{(2)}_{12} \otimes e^{(p)}_{11})=\lambda (e^{(2)}_{12} \otimes e^{(p)}_{11})$. Thus $A$ has a non-zero element $a$ with $D(a)=\lambda a$. We will call $a$ an \emph{e-element} and $\lambda$ an \emph{e-value}. Similarly, there are e-elements with e-values $0$ and $-\lambda$. Now suppose $b = \sum \beta_{ij,kl} e^{(2)}_{ij} \otimes e^{(p)}_{kl}$ is an e-vector in $A$ with e-value $\mu \neq 0$, so $D(b)=\mu b$. We calculate $D(b)$: this is 
\[
\sum \beta_{ij,kl}^\prime e^{(2)} \otimes e^{(p)}_{kl} +( [Z(\lambda)\otimes 1, \sum \beta_{ij,lk} e^{(2)}_{ij} \otimes e^{(p)}_{kl}] = \sum \beta_{ij,lk} [Z(\lambda),e^{(2)}_{ij}] \otimes e^{(p)}_{kl}).
\]
Setting $D(b)$ equal to $\mu b$ and comparing coefficients gives\
\[
\beta^\prime_{11,kl}=\mu \beta_{11,kl}, \beta^\prime_{22,kl}=\mu \beta_{22,kl}, \beta^\prime_{12,kl} + \lambda \beta_{12,kl} = \mu \beta_{12, kl}, \beta^\prime_{21,kl} - \lambda \beta_{21,kl} = \mu \beta_{21, kl}.
\]
We rewrite the last two equations as $\beta^\prime_{12,kl} =-(\lambda - \mu) \beta_{12, kl}, \beta^\prime_{21,kl}=  ( \lambda + \mu) \beta_{21, kl}$.
In the field $C(x)$ if $y^\prime =cy$ for $c \in C-\{0\}$ then $y=0$. Thus $\beta_{11,kl}=\beta_{22,kl}=0$, and either $\lambda=\mu$, $  \beta_{12, kl} \in C$ and $ \beta_{21, kl}=0$ or $\lambda=-\mu$, $  \beta_{12, kl} =0$ and $ \beta_{21, kl}\in C$. 

Thus the only non-zero e-values in $A$ are $\lambda, -\lambda$, and the corresponding e-vectors are in $(M_2(C), Z(\lambda)) \otimes (M_p(C), 0)$. We write $A=A(\lambda)$ to be explicit about the value $\lambda$,  and let $A_0(\lambda)$ denote  $(M_2(C), Z(\lambda))$.
Finally, let $\lambda_1$ and $\lambda_2$ be non-zero elements of $C$ and suppose that $ \pm \lambda_1 \neq \lambda_2 $. If $A_0(\lambda_i)$, $i=1,2$ go to the same thing under $BM^\text{diff}(C) \to BM^\text{diff}(C(x))$ then there is a $p$ such that $A(\lambda_1)$ is differentially isomorphic to $A(\lambda_2)$. But then the above calculation on e-values shows that  either $ \lambda_1 = \lambda_2 $ or  $ \lambda_1 = -\lambda_2 $, contrary to assumption. Thus the $A_0(\lambda)$ represent distinct elements of $BM^\text{diff}(C(x))$ (and \emph{a fortiori} different elements of $BM^\text{diff}(C)$).

\begin{example} \label{PVring} We consider the case that $F$ is an algebraically closed differential field with $C=F^D$. \end{example} Let $R$ be the Picard--Vessiot ring of a differential Galois extension $E$ of $F$ with differential Galois group $G$, which is a connected linear algebraic group over $C$. Then also $E^D=R^D=C$. We suppose the extension is given by a matrix differential equation $X^\prime=AX$ for $A \in M_n(F)$, $n > 1$. Then $E$ contains, the entries of, and is generated by the entries of, an invertible matrix $Y \in M_n(R)$ with $Y^\prime =AY$. This means that $(M_n(F), -A)$ goes to the trivial class in $BM^\text{diff}(F) \to BM^\text{diff}(R)$. (If $(M_n(F),-A)$ were already trivial then the Picard--Vessiot extension would be trivial.) Since $F$ is algebraically closed, $R$, as a differential ring, is isomorphic to $F[G]$. 

If $G$ is unipotent, then $R$ is a polynomial ring so all projective $R$ modules are free, so $BM(R)=Br(R)$, and since $R$ has trivial Brauer group, $BM(R)=1$. Thus by Corollary \ref{trivialBM} $BM^\text{diff}(R)=MBM^\text{diff}(R)$. 

If $G$ is a torus of rank $m$, $R$ is a Laurent polynomial ring, so all projective $R$ modules are free so $BM(R)=Br(R)$. Since $R^D=C$, $BM(R^D)=1$, so by 
Theorem \ref{unitsinduced}
$BM^\text{diff}(R)$ has no units. In this case $Br(R)$ is a product of $m(m-1)/2$ copies of $\mathbb Q/\mathbb Z$ \cite[Corollary 7, p. 166]{m} so $MBM^\text{diff} 
\subsetneq BM^\text{diff}(R)$.

\begin{example} \label{completePVring} Let $F$ be a differential field with $F^D=C$ and let $R$ be the Picard--Vessiot ring of a complete Picard--Vessiot compositum $E$ of $F$. \end{example}
We recall that a complete Picard--Vessiot compositum $E$ of $F$ is a differential field extension of $F$ with the same constants as $F$, generated over $F$ by the Picard--Vessiot extensions of $F$ contained in $E$, into which every Picard--Vessiot extension of $F$ maps. The Picard--Vessiot ring $R$ of $E$ is the set of elements satisfying monic linear homogeneous differential equations over $F$. Then for every $n$ and every $A \in M_n(F)$ a Picard--Vessiot extension for $X^\prime=AX$ appears in $E$, and thus there is an invertible $Y \in M_n(R)$ with $Y^\prime=-A Y$, which means that $(M_n(F), A)$ goes to the trivial class under $BM^\text{diff}(F) \to BM^\text{diff}(R)$. (Neither monoid has any units since $BM(F^D)=BM(R^D)=BM(C)=1$.) In particular, under $BM^\text{diff}(F) \to BM^\text{diff}(R)$,  $MBM^\text{diff}(F)$ goes to the identity.

A complete Picard--Vessiot compositum may have proper Picard--Vessiot extensions. By taking a complete Picard--Vessiot compositum of a complete Picard--Vessiot compositum, and iterating the process a countable number of times, one obtains a differential extension $\mathcal E$ of $F$ such that for any differential subfield $F_1$ of $\mathcal E$ every Picard--Vessiot extension of $F_1$ maps over $F_1$ to $\mathcal E$. That is, $\mathcal E$ contains a complete Picard--Vessiot compositum of $F_1$.
Let  $A \in M_n(\mathcal E)$ and let $F_1$ be a differential subfield of $\mathcal E$ containing the entries of $A$. Let $E_1$ be a complete Picard--Vessiot compositum of $F_1$ contained in $\mathcal E$. As we just saw, $(M_n(F), A)$ goes to the trivial class under $BM^\text{diff}(F) \to BM^\text{diff}(R_1)$, where $R_1$ is the Picard--Vessiot ring of $E_1$. Since $R_1 \subset \mathcal E$,  $(M_n(F), A)$ is also the trivial class in $BM^\text{diff}(\mathcal E)$. Hence 
$MBM^\text{diff}(\mathcal E) =1$. So by Corollary \ref{MBMtrivial} $BM^\text{diff}(\mathcal E)=BM(\mathcal E)$. Because $\mathcal E$ is a field, $BM(\mathcal E)=Br(\mathcal E)$.
Since $\mathcal E$ is algebraically closed (algebraic extensions are Picard--Vessiot), $Br(\mathcal E)=1$. Thus $BM^\text{diff}(\mathcal E)=1$.

\end{document}